\documentclass[leqno, 12pt]{article}
\usepackage{amsmath, amsfonts, amssymb}
\usepackage[latin1]{inputenc}
\usepackage[english]{babel}
\usepackage{makeidx}

\newtheorem{theorem}{Theorem}[section]

\newtheorem{proposition}[theorem]{Proposition}

\newtheorem{definition}[theorem]{Definition}

\newcommand{\R}{{\mathbb R}}

\newenvironment{proof}{\noindent{\it Proof.}}{\hfill$\Box$}

\begin{document}

\title{Resonant biharmonic problems with asymmetric-superlinear nonlinearities}

\author{Fabiana Maria Ferreira \thanks{fabiana.m.ferreira@ufes.br}\\ Departamento de Matemática Pura e Aplicada \\
Universidade Federal do Espírito Santo\\ 29500-000 -  Alegre, ES, Brazil\\
%
and \\
Wallisom Rosa\thanks{wallisom@ufu.br }\\  Instituto de Ci\^encias Exatas e Naturais do Pontal \\ Universidade Federal de Uberl\^andia \\
38304-402 -  Ituiutaba, MG, Brazil}




\date{}
\maketitle

\begin{abstract}
The aim of this work is to present results of existence of solutions for a class of biharmonic elliptic problems with homogeneous Navier conditions. The problem is asymmetric and has linear behavior on $-\infty$ and superlinear on $+\infty$. To obtain these results we apply  topological degree theory.
\end{abstract}

\noindent $\bf{Keywords.}$  biharmonic problem, topological degree, a priori estimates.


\section{Introduction}

Let $\Omega$ be a smooth bounded domain of $\R^N$, $N>5$.  Our aim is to investigate the solvability of the semilinear biharmonic problem 
\begin{align}
\label{prob:bi}
\begin{cases}
\left(-\Delta\right)^2 u= \lambda_1^2 u + u_+^p + f(x)  & x\in\Omega\\
u  = \Delta u = 0 & x\in \partial\Omega,
\end{cases}
\end{align}
where $\lambda_1$ is the first eigenvalue of $(-\Delta,H_0^1(\Omega))$, $u_+=\max\{u,0\}$ and $p>1$.   
We will assume that the function $f$ is nonzero function satisfying
the following conditions
\begin{align}
\label{hip:f}
 f\in L^r(\Omega)\ \textrm{with} \ r>\frac{N}{3},
\end{align}
and
\begin{align}
\label{hip:f0}
\int_{\Omega}{f(x)\phi_1}<0,
\end{align}
where  $\phi_1$ is the positive eigenfunction associated to $\lambda_1$.  

The assimetric nonlinearity $\lambda u+(u^+)^p$  caracterize the problem as superlinear at $+\infty$ and resonant at $-\infty$.

We will apply the Brezis-Turner technique to get a priori estimates for the solutions.  In the nonresonant case, this tecnique was sucessfull applied for a poliharmonic equation in \cite{CdeFM}.  For the resonant case,  we will follow the aproach in \cite{CFSl}, where the main operator is the Laplacian.  Then we show:


\begin{theorem}
\label{teo:existencia}
Let $f$ satisfy \eqref{hip:f} and \eqref{hip:f0}, and suppose that 
\begin{align}
\label{hip:p}
\max\left\{1,\frac{4}{N-4}\right\}<p<\frac{N+1}{N-3}.
\end{align}
Then problem \eqref{prob:bi} admits at least one solution $ u \in W^{4,r}\cap W^{2,r}_0(\Omega)$.
\end{theorem}

We point out that the condition \eqref{hip:f} implies, by regularity theory, that all weak solutions $u\in H_0^1(\Omega)\cap H^2(\Omega)$ of
\eqref{prob:bi} belong to $W^{4,r}(\Omega)$, and recall that $W^{4,r}(\Omega)\subset C^1(\Omega)$ because $r>N/3$. 

Recently, there has been increasing interest  in study elliptic problem with asymmetric nonlinearities of the type: asymptotically linear at $-\infty$ and superlinear at $+\infty$.   

We emphasize the existence of a priori bounds of solution is very important for to show the existence of solutions
when the problem in question is not of variational types. 
The techniques used here were inspired by the following works \cite{BeTu}, \cite{CdeFM} and \cite{CFSl}. 
In article  \cite{CFSl}, the authors obtain a solution of the following superlinear elliptic problem 
\begin{align*}
\begin{cases}
-\Delta u =  \lambda_1 u + u^p_+ + f(x)  & x\in\Omega\\
u =0  & x\in \partial\Omega,
\end{cases}
\end{align*}
under the assumptions
$1<p<\frac{N+1}{N-1}$, $f\in L^r$ with $r>N$  and $\int_{\Omega}{f\phi_1}<0$.  The proof of Theorem \ref{teo:existencia} uses the thecnique introduced in \cite{BeTu}.  The method consists in getting a priori bounds, using Hardy-Sobolev type inequalities, with topological degree arguments.   
Similar problems, with Dirichelt and Neumann boundary condition, can be found in \cite{bib:KaOr1,bib:Ward}.  

The Neumann scalar equation:
\begin{equation}\label{P4intro}
\left\{\begin{array}{lr}
-\Delta u=(u^+)^p+f(x),& x\in\Omega,\\ 
\displaystyle\frac{\partial u}{\partial \nu}=0,& x\in\partial\Omega,
\end{array}\right.
\end{equation}
with $f$ satisfying (3), was studied in the paper \cite{bib:Ward}.  The author applied a continuation theorem due to Mawhin and others results due to Brezis and Strauss.   

In the work \cite{WallyOdair} was studied the Hamiltonian system:
\begin{equation}\label{P4intro2}
\left\{\begin{array}{lr}
\displaystyle -\Delta u=(v^+)^p+f(x), & \hbox{em}\ \ \Omega,\\
\displaystyle -\Delta v=(u^+)^q+g(x), & \hbox{em}\ \ \Omega,\\
\displaystyle \frac{\partial u}{\partial \nu}=\frac{\partial v}{\partial\nu}=0,& \hbox{em}\ \ \partial\Omega.
\end{array}\right.
\end{equation}
There also was assumed that $f$ and $g$ satisfy (3). Note that, with Neumann conditions at the border, $\varphi_1$ is constant (so has signal set) and so we must to assume that $f$ and $g$ have strictly negative integral in $\Omega$. As well as in paper \cite{CuestaDjairo}, the authors used topological methods for obtaining the results of existence of solutions. The theory of the degree of Leray-Schauder was an essential tool in this process, also they used results of \cite{bib:KaOr2} to obtain the essential a priori estimates about solutions for the system  (\ref{P4intro2}).

We also mention that in \cite{CFSl}, the authors consider under appropriate conditions  a polyharmonic equation 
\begin{align}
\begin{cases}
\left(-\Delta\right)^m u = h(x,u, \Delta u, \cdots , \Delta^{m-1} u)  & x\in\Omega\\
u  = \Delta u = \cdots = \Delta^{m-1} u = 0 & x\in \partial\Omega.
\end{cases}
\end{align}

We will  denote the usual norm of the Sobolev Space $W^{k,p}(\Omega)$ by $\left\|\cdot\right\|_{k,p}$.
The space $C^1_0(\overline{\Omega})$ is defined as
 $$C^1_0(\overline{\Omega})=\{u\in C^1(\overline{\Omega}); u=0 \ \textrm{on} \ \partial \Omega\}$$
which was a Banach space with the norm
 $$\left\|u\right\|_{C^1_0(\overline{\Omega})}=\max_{x\in\overline{\Omega}}\left|u(x)\right|+\max_{x\in\overline{\Omega}}\left|\nabla u(x)\right|.$$

\section{A priori estimates}

First  is necessary to remember the following lemma based on the Hardy-Sobolev inequality.  The proof  can be found in \cite{CdeFM}.

\begin{proposition}
\label{cor:m}
Let $u\in W_0^{1,s}(\Omega)\cap W^{4,s}(\Omega)$ with $1<s<\frac{N}{4}$ and $\tau \in [0,1]$. If
 $$\frac{1}{r} = \frac{1}{s} - \frac{4-\tau}{N}, $$
then
\begin{align}
\label{desig:corm}
\left\|\frac{u}{\phi_1^{\tau}}\right\|_{L^r} \leq C \left\|u\right\|_{W^{4,s}}
\end{align}
where the constant  $C$ depends only on  $\tau$, $s$ and $N$
\end{proposition}

Now we can used that to prove a priori bounds for the biharmonic problem \eqref{prob:bi}.

\begin{theorem}
\label{lem:estimativa}
Let $u\in H_0^1(\Omega)\cap H^2(\Omega)$ be a solution of problem \eqref{prob:bi}. 
 Under the assumptions of Theorem \ref{teo:existencia}, 
there exist an increasing continuous function
$\rho: \mathbb{R}^+ \rightarrow  \mathbb{R}^+ $,
depending only on  $p$ e $\Omega$, such that $\rho\left(0\right)=0$ and 
\begin{align}
\label{desig:rho}
\left\|u\right\|_{C_0^1(\bar{\Omega})}  \leq \rho (\left\|f\right\|_{r}).
\end{align}
\end{theorem}

\begin{proof}
Let $u\in H_0^1(\Omega)\cap H^2(\Omega)$ be a weak solution of \eqref{prob:bi}. By  multiplying  \eqref{prob:bi} 
by eingefunction $\phi_1$ we find that
\begin{align*}
\int_{\Omega}{(-\Delta)^2 u \phi_1} = \int_{\Omega}{\lambda_1^2 u \phi_1}  + \int_{\Omega}{u_+^p\phi_1} + \int_{\Omega}{f\phi_1},
\end{align*}
then
\begin{align}
\label{desig:bit}
\int_{\Omega}{u_+^p\phi_1}= -\int_{\Omega}{f\phi_1}\leq C\left\|f\right\|_{r}.
\end{align}
Now we write $u=t \phi_1+ u_1$ where  $u_1$ and $\phi_1$ orthogonal in space $H_0^1(\Omega)$ 
therefore $\int_{\Omega}{u_1\phi_1}=0$.
Multiplyng this decomposition for  $\phi_1$ and  integration, result 
\begin{align*}
\int_{\Omega}{u\phi_1}	=  t\int_{\Omega}{\phi_1^2} + \int_{\Omega}{u_1\phi_1},
\end{align*}
then
\begin{align*}
t =  C\int_{\Omega}{u\phi_1}	=  C\int_{\Omega}{(u_{+}-u_-)\phi_1}   \leq  C \left(\int_{\Omega}{u_+^p\phi_1}\right)^{1/p},
\end{align*}
and using the inequality \eqref{desig:bit}, we conclude 
\begin{align}
\label{desig:tt}
t\leq C\left\|f\right\|_{r}^{1/p}.
\end{align}
After that, we broke the proof in two parts, according to the sign of $t$.


\textbf{Case 1:} $t\geq 0$. In this case the equation \eqref{desig:tt} give us a estimate for  $\left|t\right|$.
So we have to find an estimate for $u_1$. 
Observe that $u_1$ satisfies the  equality
\begin{align*}
\int_{\Omega}{\left(-\Delta\right)^2 u_1} = \int_{\Omega}{\lambda_1^2 u_1} + \int_{\Omega}{u_+^p} + \int_{\Omega}{f(x)},
\end{align*}
and by taking the $L^{\frac{p+1}{p}}$-norm on both sides result
\begin{align}
\label{desig:bioper}
\int_{\Omega}{\left|\left(-\Delta\right)^2 u_1 - \lambda_1^2 u_1\right|^{\frac{p+1}{p}}} \leq \int_{\Omega}{u_+^{p+1}} 
+ \int_{\Omega}{\left|f(x)\right|^{\frac{p+1}{p}}}.
\end{align}
We can write the first integral in of the right in \eqref{desig:bioper}  following way
\begin{align*}
\int_{\Omega}{u_+^{p+1}}  = \int_{\Omega}{u_+^{p\alpha}\phi_1^{\alpha}\phi_1^{-\alpha}u_+^{p(1-\alpha)+1}}
\end{align*}
with $0<\alpha<1$ to be determined later. Using to the inequality of  $H\ddot{o}lder$ and the inequality \eqref{desig:bit}, we obtain
\begin{align}
\label{desig:biestu}
\int_{\Omega}{u_+^{p+1}} & \leq  \left(\int_{\Omega}{u_+^{p}\phi_1 }\right)^{\alpha} 
\left(\int_{\Omega}{\frac{u_+^{p+\frac{1}{1-\alpha}}}{\phi_1^{\frac{\alpha}{1-\alpha}}}}\right)^{1-\alpha}
 \leq  \left\|f\right\|_r^{\alpha}\left(\int_{\Omega}{\frac{u_+^{p+\frac{1}{1-\alpha}}}{\phi_1^{\frac{\alpha}{1-\alpha}}}}\right)^{1-\alpha}.
\end{align}
Now let's review at the second integral on the right side of the \eqref{desig:bioper}. Using the hiphotesis \eqref{hip:p} and the fact that $\displaystyle r>\frac{N}{3}$ we obtained that  $\displaystyle\frac{p+1}{p}\frac{1}{r}<1$. Due to H\"{o}lder's inequality, it follows that
\begin{align}
\label{desig:biestf}
\int_{\Omega}{\left|f(x)\right|^{\frac{p+1}{p}}} \leq C\left\|f(x)\right\|_{r}^{\frac{p+1}{p}}.
\end{align}
Replacing \eqref{desig:biestu} and \eqref{desig:biestf} in \eqref{desig:bioper}, result
\begin{align}
\label{desig:tbvale}
\left\|\left(-\Delta\right)^2 u_1 - \lambda_1^2 u_1\right\|^{\frac{p+1}{p}}_{\frac{p+1}{p}} \leq
\left\|f(x)\right\|_r^{\alpha}\left(\int_{\Omega}{\frac{u_+^{p+\frac{1}{1-\alpha}}}{\phi_1^{\frac{\alpha}{1-\alpha}}}}\right)^{1-\alpha}
+ \left\|f(x)\right\|^{\frac{p+1}{p}}_{r}.
\end{align}
Now we use the Proposition \ref{cor:m} for estimate the term
\begin{align*}
I=\left(\int_{\Omega}{\frac{u_+^{p+\frac{1}{1-\alpha}}}{\phi_1^{\frac{\alpha}{1-\alpha}}}}\right)^{1-\alpha},
\end{align*}
we may to write the integral  I as follows 
\begin{align*}
I=\left\|\frac{u_+}{\phi_1^{\tau}}\right\|^{p(1-\alpha)+1}_{t},
\end{align*}
with
\begin{align}
\label{ig:ttau}
 t= p+\frac{1}{1-\alpha}  \quad \textrm{and} \quad \tau t=\frac{\alpha}{1-\alpha}.
\end{align}
Define 
\begin{align*}
L = \frac{p}{p+1}-\frac{4}{N},
\end{align*}
note that  $L>0$ because inequality \eqref{hip:p} implies that $\displaystyle 1<\frac{p+1}{p}<\frac{N}{4}$.
Now our aim is to use the Proposition \ref{cor:m}
for  $\displaystyle s=\frac{p+1}{p}$. For  this we must find $\alpha\in\left(0,1\right)$ such that
\begin{align*}
\frac{1}{t}=\frac{1}{s}-\frac{4-\tau}{N}.
\end{align*}
These equality and  equations \eqref{ig:ttau} determine 
\begin{align}
\label{ig:des}
\alpha = \frac{N-NL -NLp}{1+N-pLN}.
\end{align}
Moreover, we can write 
\begin{align}
\label{ig:des}
 \tau=\frac{N-NL -NLp}{1+N+p}, \quad t=\frac{1+N+p}{1+NL},
\end{align}
and using this values it is possible to show that $\alpha \in(0,1)$ e $\tau\in [0,1]$.
Applying the Proposition \ref{cor:m} we get
\begin{align*}
I\leq \left\|u\right\|^{p(1-\alpha)+1}_{4,\frac{p+1}{p}},
\end{align*}
and replacing this inequality in \eqref{desig:tbvale}, we obtain
\begin{align}
\label{desig:iso}
\left\|\left(-\Delta\right)^2 u_1 - \lambda_1^2 u_1\right\|^{\frac{p+1}{p}}_{\frac{p+1}{p}} \leq
\left\|f(x)\right\|^{\alpha}_{r}\left\|u\right\|^{p(1-\alpha)+1}_{4,\frac{p+1}{p}}
+ \left\|f(x)\right\|^{\frac{p+1}{p}}_{r}.
\end{align}
Now, we use the fact that operator
$\left(-\Delta\right)^2  - \lambda_1^2 Id: W^{4,\frac{p+1}{p}}_{*}(\Omega) \longrightarrow L^{\frac{p+1}{p}}_*(\Omega)$
is an isomorphism where
$$ W^{4,\frac{p+1}{p}}_{*}(\Omega)= \left\{ u\in  W^{4,\frac{p+1}{p}}(\Omega); \int_{\Omega}{u \phi_1}=0\right\}$$
and
$$L^{\frac{p+1}{p}}_*(\Omega)= \left\{ u\in  L^{\frac{p+1}{p}}(\Omega) ; \int_{\Omega}{u \phi_1}=0\right\}.$$ 
Therefore, exist $c>0$ such that
\begin{align*}
\left\|u_1\right\|_{4,\frac{p+1}{p}}  \leq c \left\|\left(-\Delta\right)^2 u_1 - \lambda_1^2 u_1\right\|_{\frac{p+1}{p}}.
\end{align*}
Replacing the previous inequality in \eqref{desig:iso}, result
\begin{align*}
\left\|u_1\right\|_{4,\frac{p+1}{p}}^{\frac{p+1}{p}}  \leq \left\|f(x)\right\|^{\alpha}_{r}\left\|u\right\|^{p(1-\alpha)+1}_{W^{4,\frac{p+1}{p}}}
+ \left\|f(x)\right\|^{\frac{p+1}{p}}_{r}.
\end{align*}
Using the decomposition $u=t\phi_1 + u_1$ and \eqref{desig:tt} we get that
\begin{align}
\label{desig:um}
\left\|u_1\right\|_{4,\frac{p+1}{p}}^{\frac{p+1}{p}}  
& \leq  \left\|f\right\|_{r}^{\alpha+ \frac{P(1-\alpha)+1}{p}} + \left\|f(x)\right\|^{\alpha}_{r}\left\|u_1\right\|^{p(1-\alpha)+1}_{W^{4,\frac{p+1}{p}}}
+ \left\|f(x)\right\|^{\frac{p+1}{p}}_{r}.
\end{align}
Using the hypothesis \eqref{hip:p}  we get to show that
\begin{align*}
\frac{1}{\theta}:=\left[p\left(1-\alpha\right)+1\right]\frac{p}{p+1}<1.
\end{align*}
Then 
applying  Young inequality in \eqref{desig:um}, result
\begin{align*}
\left\|u_1\right\|_{4,\frac{p+1}{p}}^{\frac{p+1}{p}}  \leq C\left( \left\|f\right\|_{r}^{\alpha+ \frac{P(1-\alpha)+1}{p}} +
\left\|f(x)\right\|^{\alpha\theta'}_{r}
+ \left\|f(x)\right\|_r^{\frac{p+1}{p}}\right),
\end{align*}
where $\theta'$ is the conjugate of the $\theta$. Finally using the decomposition  $u=t\phi_1+u_1$, we obtain
\begin{align}
\label{desig:bibi}
\left\|u\right\|_{4,\frac{p+1}{p}}^{\frac{p+1}{p}}  \leq C\left( \left\|f\right\|_{r}^{\alpha+ \frac{P(1-\alpha)+1}{p}} +
\left\|f(x)\right\|_r^{\alpha\theta'}
+ \left\|f(x)\right\|_r^{\frac{p+1}{p}} + \left\|f\right\|_r^{1/p}\right) .
\end{align}  
A bootstrap argument and regularity theory will give  that $u\in W^{4,r}(\Omega)$ and exist  $C>0$, such that
\begin{align*}
\left\|u\right\|_{4,r} & \leq C \left( \left\|u\right\|_{4,\frac{p+1}{p}}^{\gamma} +\left\|f\right\|_{r}^{\xi} \right)
\end{align*}
with $\gamma,\xi\geq1$. Using the inequality  \eqref{desig:bibi}, we obtain
\begin{align*}
\left\|u\right\|_{4,r} & \leq \rho(\left\|f\right\|_r).
\end{align*}
And recall that $W^{4,r}(\Omega)\hookrightarrow C^1(\overline{\Omega})$ because $r>N/3$. Therefore
\begin{align*}
\left\|u\right\|_{C^1_0(\overline{\Omega})} & \leq \rho(\left\|f\right\|_r).
\end{align*}

\textbf{Case 2:} \ $t< 0$. By Hopf's Maximum Principle the first eigenfunction of $(-\Delta, H_0^1(\Omega))$, $\phi_1$ 
lies in the interior of the cone of positive functions in the space $C_0^1(\overline{\Omega})$. 
Then there exist $\epsilon>0$ such that 
$$w\in B_{C_0^1(\overline{\Omega})}(\phi_1,\epsilon) \Rightarrow w>0 \ \textrm{in}\  \Omega  \ \textrm{and} 
\ \frac{\partial w}{\partial \eta}< 0 \ \textrm{on} \ \partial \Omega.$$
where $\eta$ denotes the exterior normal derivative at the boundary of $\Omega$. Recall that  our solution
$u$ of the problem \eqref{prob:bi}, as well $u_1$ belongs to $C_0^1(\Omega)$. Let $\epsilon_0$ be the supremum of such  $\epsilon$'s.
We affirm that $-u_1/t  \notin B_{C_0^1(\Omega)}(0,\epsilon_0)$. Otherwise, it would have
\begin{align*}
\frac{u}{t}= \phi_1-\left(-\frac{u_1}{t}\right) \in B_{C_0^1(\Omega)}(\phi_1,\epsilon_0) .
\end{align*}
This is a contradiction with the fact that $u^+\neq 0$ and $t<0$. Hence,
\begin{align*}
\left|t\right|\leq \frac{1}{\epsilon_0}\left\|u_1\right\|_{C_0^1(\Omega)}. 
\end{align*}
Now, we will need find a priori bound of $u_1$ in $C_0^1(\Omega)$. Note that the inequality \eqref{desig:tbvale}
remains true for $t<0$ and this case we have $u_+\leq u_1$, hence
\begin{align*}
\left\|\left(-\Delta\right)^2 u_1 - \lambda_1^2 u_1\right\|^{\frac{p+1}{p}}_{\frac{p+1}{p}} \leq
\left\|f(x)\right\|_r^{\alpha}\left\|\frac{u_1}{\phi_1^{\tau}}\right\|^{p(1-\alpha)+1}_{L^{p+\frac{1}{1-\alpha}}}
+ \left\|f(x)\right\|_r^{\frac{p+1}{p}}
\end{align*}
Using the Proposition \ref{cor:m} and the inequality \eqref{desig:tt}, we obtain 
\begin{align*}
\left\|\left(-\Delta\right)^2 u_1 - \lambda_1^2 u_1\right\|^{\frac{p+1}{p}}_{\frac{p+1}{p}} \leq
\left\|f(x)\right\|_r^{\alpha}\left\|u_1\right\|^{p(1-\alpha)+1}_{4,\frac{p+1}{p}}
+ \left\|f(x)\right\|_r^{\frac{p+1}{p}}.
\end{align*}
Similarly to the previous case we apply Young inequality and we have
\begin{align}
\label{desig:p1p}
\left\|u_1\right\|^{\frac{p+1}{p}}_{4,\frac{p+1}{p}} \leq
C\left( \left\|f(x)\right\|_r^{\alpha\theta'} 
+ \left\|f(x)\right\|_r^{\frac{p+1}{p}}\right).
\end{align}
We now use that $u_1$ solves the problem
\begin{align*}
\begin{cases}
\left(-\Delta\right)^2 u_1= \lambda_1^2 u_1 + u_+^p + f(x)  & x\in\Omega\\
u_1  = \Delta u_1 = 0 & x\in \partial\Omega.
\end{cases}
\end{align*}
A bootstrap argument and regularity theory implies  that $u\in W^{4,r}(\Omega)$ and exist  $C>0$, such that
\begin{align*}
\left\|u\right\|_{4,r} & \leq C \left( \left\|u\right\|_{4,\frac{p+1}{p}}^{\gamma} +\left\|f\right\|_{r}^{\xi}\right),
\end{align*}
with  $\gamma,\xi>1$. Replacing  inequality \eqref{desig:p1p} in previous inequality we obtain
\begin{align*}
\left\|u\right\|_{4,r} & \leq \rho(\left\|f\right\|_{r}).
\end{align*}
And recall that $W^{4,r}(\Omega)\hookrightarrow C^1(\overline{\Omega})$ because $r>N/3$, therefore
 \begin{align*}
\left\|u\right\|_{C^1_0(\overline{\Omega})} & \leq \rho(\left\|f\right\|_{r}).
\end{align*}
\end{proof}

\section{Proof of the Theorem \ref{teo:existencia}}

First, let us introduce the following formulation of  problem \eqref{prob:bi}.
Let 
$T_{f}: C_0^1(\bar{\Omega})\longrightarrow C_0^1(\bar{\Omega})$  map such at
$$T_{f}(u) = \left(\Delta^2\right)^{-1}\left(\lambda_1^2 u + u_+^p + f\right).$$ 
The map $T_f$ is  continuous, compact  and $T_f(u)=u$, if, and only if, $u$ is a solution of  problem \eqref{prob:bi}.

To prove  Theorem \ref{teo:existencia} we use the following proposition.

\begin{proposition}
\label{prop:bi}
There exist $\epsilon>0$ and $R_0>0$ such that for all function $f\in L^{r}(\Omega)$ satisfying the condition  (\ref{hip:f0})
with $\left\|f\right\|_r<\epsilon$ and for which problem (\ref{prob:bi}) possesses at least one solution, it follows that 
$$\deg\left(Id-T_f,B_{C_0^1(\bar{\Omega})}(0,R),0\right)\neq0$$
for all $R>R_0$. 
\end{proposition}

\begin{proof}
Let $\rho$ be the function given by  Theorem \ref{lem:estimativa}, as the function $\rho$ is increasing
consider $\epsilon <1$, such that $\rho(\epsilon)< \left(\frac{\lambda_2^2-\lambda_1^2}{p}\right)^{\frac{1}{p-1}}$. 
Let $f\in L^{r}(\Omega)$ with $\left\|f\right\|_{r}\leq \epsilon$ and satisfying  the condition  \eqref{hip:f0}. By Theorem
\ref{lem:estimativa} any solution  $u_0$ of \eqref{prob:bi} satisfies 
\begin{align}
\label{desig:nor}
\left\|u_0\right\|_{C_0^1}<\left(\frac{\lambda_2^2-\lambda_1^2}{p}\right)^{\frac{1}{p-1}}.
\end{align}
Now, consider the linearization of  problem \eqref{prob:bi} at some solution $u_0$ 
\begin{align}
\label{prob:linearizacao}
\begin{cases}
\left(-\Delta\right)^2 v = \left[\lambda_1^2  + p\left(u^+_0\right)^{p-1}\right]v & x\in\Omega\\
v  = \Delta v = 0 & x\in \partial\Omega.
\end{cases}
\end{align}
Using  inequality \eqref{desig:nor}  we obtain a.e. $x\in\Omega$ that
\begin{align}
\label{desig:auts}
\lambda_1^2<\lambda_1^2  + p\left(u^+_0(x)\right)^{p-1}<\lambda_2^2
\end{align}

Denote by  $\mu_1(g)<\mu_2(g)\leq...\leq \mu_n(g)\leq...$ the  eigenvalues of the following eigenvalue problem of weight
\begin{align}
\label{prob:autpeso}
\begin{cases}
\left(-\Delta\right)^2 v= \mu g(x)v & x\in\Omega\\
v  = \Delta v = 0 & x\in \partial\Omega.
\end{cases}
\end{align}
Observe that, if $g(x)=\lambda_1^2  + p\left(u^+_0\right)^{p-1}$ we obtain the linearization \eqref{prob:linearizacao} and
from the  \eqref{desig:auts} we have that $\lambda_1^2<g(x)<\lambda_2^2$. Using the Theorem \ref{teo:contunicabi},
we conclude that the eigenfunctions associated to problem \eqref{prob:autpeso} 
satisfy the Unique Continuation Property, (See Definition \eqref{def:unica}). Then using  Proposition
\ref{prop:pcubi} we obtain
\begin{align*}
\mu_1(g) < \mu_1(\lambda_1^2) = 1 = \mu_2(\lambda_2^2) < \mu_2(g).
 \end{align*}
We conclude that the only solution of problem \eqref{prob:linearizacao} is $u_0\equiv 0$.  Therefore  $u_0$ 
is a non-degenerate solution of \eqref{prob:bi}. Using the Degree Theory, result
$$\deg\left(Id-T_f,B_{C_0^1(\bar{\Omega})}(u_0,R_0),0\right) = (-1)^1,$$
theremore, the solution $u_0$ is isolated and that there is only a finite number of them in $B_{C_0^1(\bar{\Omega})}(0,R)$.
Then
$$\deg\left(Id-T_f,B_{C_0^1(\bar{\Omega})}(0,R),0\right) = \sum{(-1)^1}\neq0,$$

\end{proof}

\subsection*{Conclusion of the proof of Theorem \ref{teo:existencia}} 
Choose $f_1 = -(t\phi_1)^p$ with $0<t<\left(\frac{\epsilon}{\left\|\phi_1^p\right\|_r}\right)^{1/p}$. 
Note that $\left\|f_1\right\|_{r}<\epsilon$ and that $u=t\phi_1$ is a solution of  problem
\eqref{prob:bi} for $f_1$ and using  Proposition \ref{prop:bi} we conclude that
$$\deg\left(Id-T_{f_1},B_{C_0^1(\bar{\Omega})}(0,R),0\right)\neq0$$
for R large enough. 

Now, consider the homotopy $H:\left[0,1\right]\times C_0^1(\Omega)\longrightarrow C_0^1(\Omega)$ such that
\begin{align*}
H\left(\tau,u\right)= \left(I-\left(\Delta^2\right)^{-1}\right)\left(\lambda_1^2u + u_+^P + \left(1-\tau\right)f + \tau f_1\right).
\end{align*}
Notice that $H\left(\tau,u\right)=0$ if, and only if, $u$ is a solution of the following problem
\begin{align}
\begin{cases}
(-\Delta)^2 u = \lambda_1^2 u + u_+^p + \left(1-\tau\right)f + \tau f_1 & \quad\mbox{em}\quad \Omega,\\
u=0  & \quad\mbox{em}\quad \partial\Omega.
\end{cases}
\end{align}
We can apply the estimate of Theorem \ref{lem:estimativa} to conclude that every solution of the problem above
 are uniformly bounded in  $C_0^1(\Omega)$ for some constant, let's say $R_1:=\rho(\max\{\left\|f\right\|_{r},\left\|f_1\right\|_{r})$. 
Hence, we have that 
\begin{equation*}
\deg(H(0,\tau), B_{C_0^1(\Omega)}(0,R_1),0)= \deg(H(1,\tau), B_{C_0^1(\Omega)}(0,R_1),0).
\end{equation*}
and for $R>0$ sufficiently large we conclude that
\begin{align*}
\deg(Id-T_{f}, B_{C_0^1(\Omega)}(0,R),0)= \deg(Id-T_{f_1}, B_{C_0^1(\Omega)}(0,R),0)\neq 0.
\end{align*}
Therefore,  $\deg(Id-T_f, B_{C_0^1(\Omega)}(0,R),0)\neq 0$ what conclude the proof of Theorem \ref{teo:existencia}.

\section{Appendix}

Here we will prove the strict monotonicity of the eigenvalues for the problem involving the 
biharmonic operator with Navier boundary condition.
For this we will use similar results proved in \cite{bib:FigueiredoD} for a  second order elliptic operator. 

Consider the problem involving the biharmonic operator
\begin{align}
\label{prob:m}
\begin{cases}
\Delta^2 u=  \mu m(x) u  & x\in\Omega\\
u  = \Delta u = 0 & x\in \partial\Omega,
\end{cases}
\end{align}
where the weight function $m(x)\in L^{\infty}(\Omega)$ and the measure of the set $\{x\in\Omega; m(x)\neq0\}$ is positive. 

We know that the eigenvalues $\mu$ of problem \eqref{prob:m} are a sequence of positive numbers
$0<\mu_1(m)<\mu_2(m)\leq...$, moreover, it has the variational characterization 
\begin{align}
\label{def:var}
\frac{1}{\mu_k(m)}= \sup_{F_k} \inf \left\{\int_{\Omega}{m u^2}; \ \ u\in F_k \ \ \text{e} \ \ \int_{\Omega}{\left|\Delta u\right|^2=1}\right\}
\end{align}
where the sets $F_k$ varies over all k-dimentional subspaces of  $H$. See \cite{bib:FigueiredoD} for more details.
We show that the strict monotonicty holds if and only if some unique continuation property is satisfied by the corresponding eigenfunctions.
The notation $\leq\neq$ means the inequality a.e. $x\in\Omega$ together with strict inequality on a set of positive measure.

\begin{definition}
\label{def:unica}
We say that a family of functions fills the Unique Continuation Property, (U.C.P.), if no function
but possibly the zero function, vanishes on a set of positive measure.
\end{definition}

\begin{proposition}
\label{prop:pcubi}
Let  $m$ and $\tilde{m}$ two weights with $m\leq\neq \tilde{m}$ and let $j\in \mathbb{Z}_0$.
If the eigenfunctions associated to $\mu_j(m)$ enjoy the unique continuation property,
then
$\mu_j(m)>\mu_j(\tilde{m})$.
\end{proposition}

\begin{proof}
Since the extrema in 
\eqref{def:var}  are achieved, there exists $F_j \in H$  of dimension $j$, such that
\begin{align}
\label{inf:variac}
\frac{1}{\mu_j(m)}=  \inf \left\{\int_{\Omega}{m u^2}; \ \ u\in F_j \ \ \text{e} \ \ \int_{\Omega}{\left|\Delta u\right|^2=1}\right\}.
\end{align}
Choose $u \in F_j$, with $\int_{\Omega}{\left|\Delta u\right|^2=1}$. Then either $u$  achieves the infimum in
 \eqref{inf:variac} or not. 

In the first case, $u$ is an eigenfunction associated to $\mu_j(m)$, and so, by the  unique continuation property, we have
\begin{align*}
\frac{1}{\mu_j(m)}= \int_{\Omega}{mu^2}< \int_{\Omega}{\tilde{m}u^2}.
\end{align*}

In the second case
\begin{align*}
\frac{1}{\mu_j(m)} < \int_{\Omega}{mu^2} \leq \int_{\Omega}{\tilde{m}u^2}.
\end{align*}

Thus, in any case
\begin{align}
\label{desig:strict}
\frac{1}{\mu_j(m)} <  \int_{\Omega}{\tilde{m}u^2}.
\end{align}
It then follows, by compactness argument, that
\begin{align*}
\frac{1}{\mu_j(m)} < \inf \left\{\int_{\Omega}{\tilde{m} u^2}; \ \ u\in F_j \ \ \text{e} \ \ \int_{\Omega}{\left|\Delta u\right|^2=1}\right\}.
\end{align*}
Therefore 
\begin{align*}
\frac{1}{\mu_j(m)}<\frac{1}{\mu_j(\tilde{m})}.
\end{align*}
\end{proof}

We will also use the following result:

\begin{theorem}
\label{teo:contunicabi}
Let $u$ solution of the  problem \eqref{prob:m} with $\left|m(x)\right|<M$. If the set $E=\left\{x \in \Omega; u\left(x\right)=0\right\}$
possess  positive measure then $u\equiv0 $ in $\Omega$. That is, the solution $u$ satisfy  the Unique Continuation Principle.
\end{theorem}
\begin{proof}
See proof in \cite{bib:porru}.

\end{proof}

\end{document}